\documentclass[a4paper,11pt]{amsart}

\usepackage{amssymb}
\usepackage[pagebackref,%colorlinks,linkcolor=black,citecolor=black%
    ,pdfborder={0 0 0}
    ,urlcolor=black,a4paper,hypertexnames=false]{hyperref}
\hypersetup{pdfauthor={Clara L\"oh, Marco Moraschini, Roman Sauer},pdftitle=Amenable covers and integral foliated simplicial volume}

\usepackage[initials, backrefs]{amsrefs}
\makeatletter
\def\print@backrefs#1{%
    \space\SentenceSpace{\color{black!60}Cited on page:~\csname br@#1\endcsname}%
}
\makeatother

\usepackage{pgf}
\usepackage{tikz}
\usetikzlibrary{decorations.pathmorphing,calc}
\tikzstyle{every picture}+=[remember picture]

\usepackage{tikz-cd}
\newtheorem{thm}{Theorem}[section]
\newtheorem{prop}[thm]{Proposition}
\newtheorem{lem}[thm]{Lemma}
\newtheorem{cor}[thm]{Corollary}

\newtheorem{question}[thm]{Question}
\newtheorem{setup}[thm]{Setup}

\theoremstyle{remark}
\newtheorem{rem}[thm]{Remark}
\newtheorem{example}[thm]{Example}

\theoremstyle{definition}
\newtheorem{defi}[thm]{Definition}

\usepackage{amsfonts}
\newcommand{\Z}{\mathbb{Z}}

\newcommand{\R}{\mathbb{R}}
\newcommand{\N}{\mathbb{N}}

\DeclareMathOperator{\tors}{tors}
\DeclareMathOperator{\rk}{rk}
\DeclareMathOperator{\mult}{mult}

\DeclareMathOperator{\id}{id}

\DeclareMathOperator{\cd}{cd_{\Z}}

\DeclareMathOperator{\cat}{cat}
\DeclareMathOperator{\Am}{Am}
\DeclareMathOperator{\amcat}{\cat_{\Am}}
\DeclareMathOperator{\cost}{cost}
\DeclareMathOperator{\rg}{rg}

\DeclareMathOperator{\im}{im}
\DeclareMathOperator{\fs}{fs}
\DeclareMathOperator{\chainb}{L^\infty_{\fs}}
\DeclareMathOperator{\Linf}{L^\infty}

\DeclareMathOperator{\ess}{ess}

\def\esst#1#2{%
  #1^{\langle #2\rangle}}

\def\lfsz#1{%
  \chainb (#1,\Z)}

\def\fa#1{%
  \forall_{#1}\;\;\;}
\def\exi#1{%
  \exists_{#1}\;\;\;}

\def\args{\;\cdot\;}

\def\longrightarrow{\rightarrow}
\def\longmapsto{\mapsto}

\def\ucov#1{%
  \widetilde{#1}
}

\def\symmdiff{\mathbin{\triangle}}

\def\linfz#1{\Linf(#1,\Z)}

\def\actson{\curvearrowright}

\def\qand{\quad\text{and}\quad}

\makeatletter
\newcommand\norm{\bBigg@{0.8}}
\makeatother

\newcommand{\ifsv}[2][norm]{\csname #1l\endcsname\bracevert\!#2\!%
                            \csname #1r\endcsname\bracevert}
\newcommand{\ifsvp}[3][norm]{\csname #1l\endcsname\bracevert\!#2\!%
                            \csname #1r\endcsname\bracevert\!^{#3}}

\def\stisv#1{\|#1\|_\Z^\infty}

\newcommand{\essn}[2][norm]{\csname #1l\endcsname\vert #2 \csname#1r\endcsname\vert_{1,\ess}}

\usepackage{color}
\usepackage{pdfcolmk}

\def\draftinfo{}%{\color{red}\ -- \textsf{This is a preliminary version!}}

\author[C.~L\"oh]{Clara L\"oh}
\address{Fakult\"at f\"ur Mathematik\\
         Universit\"at Regensburg\\
         93040~Regensburg\\
         }
\email{clara.loeh@mathematik.uni-r.de}

\author[M.~Moraschini]{Marco Moraschini}
\address{Dipartimento di Matematica\\ Universit\`{a} di Bologna\\ 40126~Bologna}
\email{marco.moraschini2@unibo.it}

\author[R.~Sauer]{Roman Sauer}

\address{Karlsruhe Institute of Technology\\
76131~Karlsruhe\\}
\email{roman.sauer@kit.edu}

\title[Amenable covers and integral foliated simplicial volume]{Amenable covers\\ and integral foliated simplicial volume}
\date{\today.\ \copyright{\ C.~L\"oh, M.~Moraschini, R.~Sauer 2021}. 
    This work was supported by the CRC~1085 \emph{Higher Invariants} 
    (Universit\"at Regensburg, funded by the DFG) and by the RTG 2229
    \emph{Asymptotic Invariants and Limits of Groups and Spaces}
    (KIT, funded by the~DFG). 
    \draftinfo\\
     MSC~2010 classification: 55N10, 55N35, 28D15} % 2do! more?! 
\begin{document}

\begin{abstract}
  In analogy with ordinary simplicial volume, we
  show that integral foliated simplicial volume
  of oriented closed connected aspherical $n$-manifolds
  that admit an open amenable cover of multiplicity
  at most~$n$ is zero. 
  
  This implies that the fundamental groups of such manifolds have
  fixed price and are cheap as well as reproves some statements about
  homology growth.
\end{abstract}

\maketitle

%%%%%%%%%%%%%%%%%%%%%%%%%%%%%%%%%%%%%%%%%%%%%%%%%%%%%%%%
\section{Introduction}

An open problem by Gromov~\citelist{\cite{gromovasym}*{p.~232}\cite{gromov-cycles}*{p.~769~3.1~(e)}}
 asks whether oriented closed
connected manifolds with vanishing simplicial volume have vanishing
Euler characteristic. For the integral foliated simplicial volume such
an implication holds~\cite{mschmidt} (as suggested by
Gromov~\cite{gromov_metric}*{p.~305ff}). However, the corresponding
follow-up question is open:

\begin{question}\label{quest:ifsv:vanishes}
  Let $M$ be an oriented closed connected aspherical manifold 
  with vanishing simplicial volume. Does then also the integral
  foliated simplicial volume of~$M$ vanish?
\end{question}

This question is known to have an affirmative answer in many
cases~\cites{FFL, FLPS, FauserS1, Camp-Corro}. 
In this article, we extend the class of positive examples
by aspherical manifolds that admit amenable open covers of
small multiplicity: If $M$ admits an open amenable cover
of multiplicity at most~$n$, we write~$\amcat M \leq n$.
Usually the amenable category $\amcat$ of $M$ is defined
in terms of the cardinality of amenable covers. 
However, in the case of CW-complexes the two definitions are equivalent~\citelist{\cite{CLM}*{Remark~3.13}\cite{CLOT}*{Lemma~A.4}}.

\begin{thm}\label{thm:main}
  Let $M$ be an oriented closed connected smooth aspherical
  manifold with~$\amcat M \leq \dim M$.
  Let $\alpha \colon
  \pi_1(M) \actson (X,\mu)$ be an essentially free standard
  $\pi_1(M)$-space.
  Then
  \[ \ifsvp M \alpha = 0.
  \]
  In particular, $\ifsv M = 0$.
\end{thm}

This is an extension of Gromov's vanishing result: 
Every oriented closed connected $n$-manifold~$M$ with
$\amcat M \leq \dim M$ has zero simplicial volume~\cite{vbc}.
The bound on the amenable
category is optimal because every smooth 
manifold~$M$ satisfies $\amcat M \leq \dim M + 1$~\cite[Remark~2.8]{CLM}.

A similar statement was already known for the weightless version of
integral foliated simplicial volume~\cite{sauerminvol} and in the
presence of macroscopic scalar curvature conditions~\cite{braunsauer}.
Our proof of Theorem~\ref{thm:main} also relies on a Rokhlin lemma
argument, but avoids delicate volume estimates in cubical nerves.

In the context of Gromov's question, we arrive at
the following problem:

\begin{question}
  Do there exist oriented closed connected aspherical manifolds~$M$ 
  with~$\|M\| = 0$ but~$\amcat M = \dim M + 1$\;?
\end{question}

We expect the answer to this question to be positive. However,
exhibiting such an example probably requires finding new obstructions
against the existence of open amenable covers of small multiplicity.

Further information on Gromov's problem and its ramifications
can be found in the literature~\cite{loehmoraschiniraptis}.

%%%%%%%%%
\subsection{Examples}\label{sec:intro:ex}

Oriented closed connected smooth manifolds~$M$ that satisfy at least
one of the following conditions have~$\amcat M \leq \dim M$:
\begin{itemize}
\item Manifolds with amenable fundamental group;
\item Graph $3$-manifolds~\cite{GGW};
\item Manifolds that are the total space of a fibre bundle $M \to B$
with oriented closed connected fibre $N$ such that
$\amcat N \leq \dim(M)\slash(\dim(B) +1)$~\cite{lmfibre};
\item Manifolds of dimension~$n \geq 4$ whose fundamental
  group~$\Gamma$ contains an amenable normal subgroup~$A$
  with~$\cd (\Gamma/A) < n$ (Lemma~\ref{lemma:cover:cohom:dimension});
\item Manifolds that admit a smooth $S^1$-action without fixed points (Corollary~\ref{cor:S1:action});
\item Manifolds that admit a regular smooth circle foliation with finite holonomy groups (Proposition~\ref{prop:foliation});
\item Manifolds that admit an $F$-structure (of possibly zero rank)~\cite{bsfibre}*{Corollary~2.8}.
\end{itemize}

In view of Theorem~\ref{thm:main}, \emph{aspherical} manifolds of this
type have vanishing integral foliated simplicial volume; more
precisely, we have vanishing for all essentially free parameter
spaces. In most of these cases, the corresponding vanishing results
already had been proved by alternative
means~\cite{FLPS,FLMQ,FauserS1,Camp-Corro}. Theorem~\ref{thm:main}
gives a uniform perspective.

%%%%%%%%%
\subsection{Application to manifolds}

Theorem~\ref{thm:main} gives new examples of manifolds
that satisfy integral approximation for simplicial
volume~\citelist{\cite{FFL}\cite{FLMQ}\cite{loehergodic}\cite{lueckl2}*{Section~6}}:

\begin{cor}\label{cor:intapprox}
  Let $M$ be an oriented closed connected smooth aspherical
  manifold with~$\amcat M \leq \dim M$ and residually finite
  fundamental group.
  Then
  \[ \stisv M = 0 = \|M\|.
  \]
  More precisely, in this situation, the corresponding statement also
  holds for all residual chains in the fundamental group.
\end{cor}
\begin{proof}
  As $\pi_1(M)$ is residually finite, the profinite
  completion~$\widehat{\pi_1(M)}$ endowed with its normalized Haar measure and the action by left translations constitutes a free
  standard $\pi_1(M)$-space, and we have~\citelist{\cite{FLPS}*{Theorem~2.6}\cite{LP}*{Theorem~6.6 and Remark~6.7}}
  \[ \stisv M = \ifsvp M {\widehat{\pi_1(M)}}.
  \]
  We can now apply Theorem~\ref{thm:main} and the general
  estimate~$\|M\| \leq \stisv M$.

  This also works in the situation of general residual
  chains~\cite{FLPS}*{Theorem~2.6}.
\end{proof}

%%%%%%%%%
\subsection{Applications to homology growth and groups}

From The\-o\-rem~\ref{thm:main}, we obtain corresponding vanishing results
for the $L^2$-Betti
numbers~\citelist{\cite{gromov_metric}*{p.~307}\cite{mschmidt}*{Corollary~5.28}}
(and whence the Euler characteristic), logarithmic torsion
growth~\citelist{\cite{FLPS}*{Theorem~1.6}\cite{sauergrowth}}, cost,
and the rank gradient~\cite{loehrg}.

The cost of a group~\cite{gaboriau_cost} is a dynamical version of the
rank gradient, which is the gradient invariant associated with the
minimal number of generators~\cite{lackenby}.

\begin{cor}\label{cor:cheap}
  Let $M$ be an oriented closed connected aspherical
  manifold with~$\amcat M \leq \dim M \neq 0$.
  \begin{enumerate}
  \item Then $\pi_1(M)$ is of fixed price and cheap.
    Thus, $\cost \pi_1(M) = 1$.
  \item If, in addition, $\pi_1(M)$ is residually finite,
    then $\rg (\pi_1(M),\Gamma_*) =0$ for all Farber chains~$\Gamma_*$
    of~$\pi_1(M)$. In particular, $\rg \pi_1(M) = 0$.
  \end{enumerate}
\end{cor}
\begin{proof}
  The first part follows from Theorem~\ref{thm:main} and the
  cost estimate via integral foliated simplicial volume~\cite{loeh_cost}.
  The second part follows from the first part and the computation of the
  rank gradient via cost~\cite{abertnikolov}.
\end{proof}

We can view Corollary~\ref{cor:cheap} as an extension of the
vanishing results for cost and rank gradients of amalgamated
free products of amenable groups~\cites{gaboriau_cost,pappas}.

Moreover, we obtain a new proof of vanishing for homology growth and
logarithmic torsion growth~\cite[Theorem~1.6]{sauergrowth}. Note that
the statement for $\mathbb{F}_p$-coefficients there is equivalent to
the statement for all principal ideal domains.  Via the bounds of
homology growth by the stable integral simplicial
volume~\cite{FLPS}*{Theorem~1.6} and Theorem~\ref{thm:main}
we obtain:

\begin{cor}
  Let $M$ be an oriented closed connected aspherical manifold
  with~$\amcat M \leq \dim M \neq 0$ and residually finite
  fundamental group. Let $(\Gamma_j)_{j \in \N}$
  be a residual chain in~$\Gamma := \pi_1(M)$. Moreover, let $k \in \N$.
  Then:
  \begin{enumerate}
  \item If $R$ is a principal ideal domain, then
    \[ \limsup_{j \to \infty} \frac{\rk_R H_k(\Gamma_j;R)}{[\Gamma:\Gamma_j]} = 0.
    \]
  \item We have
    \[ \limsup_{j \to\infty} \frac{\log \;\bigl| \tors H_k(\Gamma_j;\Z) \bigr|}{[\Gamma:\Gamma_j]} = 0.
    \]
  \end{enumerate}
\end{cor}

%%%%%%%%%%%
\subsection*{Organisation of this article}

Section~\ref{sec:prelim} introduces basic terminology.
An outline of the proof of Theorem~\ref{thm:main} is
given in Section~\ref{sec:outline}. The Rokhlin chain
maps are developed in Section~\ref{sec:rokhlin}.
The topological steps are carried out in Section~\ref{sec:subdiv}.
Finally, Section~\ref{sec:proof} wraps up the proof of Theorem~\ref{thm:main}.
Section~\ref{sec:exa} contains examples of manifolds with
small~$\amcat$.

%%%%%%%

\subsection*{Acknowledgements}

We would like to thank Caterina Campagnolo for useful discussions about foliations.
We are grateful to the anonymous referee for the detailed and helpful comments.

%%%%%%%%%%%%%%%%%%%%%%%%%%%%%%%%%%%%%%%%%%%%%%%%%%%%%%%%%%%%%%%
\section{Preliminaries}\label{sec:prelim}

We recall basic notions on integral foliated simplicial volume,
amenable open covers, and resolutions.

%%%%%%%%%
\subsection{Simplicial volume}

The simplicial volume is a homotopy invariant of closed manifolds introduced 
by Gromov~\cite{vbc}. Let $M$ be an oriented closed connected 
$n$-manifold and let $R$ be a normed ring with unit, e.g., $\R$ or~$\Z$. 
For a singular $R$-chain $c = \sum_{i = 1}^k a_i \sigma_i \in C_n(M; R)$, 
one defines the \emph{$\ell^1$-norm} by 
$$
|c|_1 := \sum_{i = 1}^k |a_i|.
$$
This norm induces a seminorm on singular homology:
Given $\alpha \in H_n(M; R)$, the \emph{$\ell^1$-seminorm} of $\alpha$
is defined as
$$
\|\alpha\|_1 := \inf \bigl\{|c|_1 \bigm| \text{$c \in C_n(M; R)$  is a cycle representing~$\alpha$}\bigr\}.
$$
By taking the $\ell^1$-seminorm of a preferred class in homology we obtain
the definition of simplicial volume:
\begin{defi}[$R$-simplicial volume]
Let $M$ be an oriented closed connected $n$-manifold. 
The \emph{$R$-simplicial volume} of $M$ is defined as
$$
\|M \|_R := \bigl\|[M]_R \bigr\|_1,
$$
where $[M]_R \in H_n(M; R)$ denotes the $R$-fundamental class of $M$.

When $R = \R$ we will simply write $\|M\| := \|M\|_\R$ and we will
refer to it simply as the \emph{simplicial volume of $M$}.
\end{defi}

\begin{rem}\label{rem:isv:bigger:1}
By definition, we have
$\|M\| \leq \|M\|_{\Z}$ and $1 \leq \|M\|_{\Z}$ (if $M$ is non-empty).
\end{rem}

Simplicial volume is known to be positive in many cases:
e.g. hyperbolic manifolds~~\cites{Thurston,
    vbc}, locally symmetric spaces of
    non-compact type~\cites{Bucher:lss, lafontschmidt}
    and manifolds with sufficiently negative curvature~\cites{inoueyano,
    connellwang}.
    Two major sources of vanishing of simplicial volume are~\cite{vbc}:
    \begin{itemize}
    \item If $M$ is an oriented closed connected $n$-manifold
      with~$\amcat M \leq \dim M$, then $\|M\| = 0$.
    \item If $M$ is an oriented closed connected $n$-manifold
      that admits a self-map~$f$ with~$|\deg f| \geq 2$,
      then $\|M\| = 0$.
    \end{itemize}
    A more comprehensive list of examples can be found in
    the literature~\cite{loehmoraschiniraptis}.

    In general integral simplicial volume and
    simplicial volume are far from being equal
    (Remark~\ref{rem:isv:bigger:1}). The situation becomes
    more interesting after stabilisation~\cites{gromov_metric,FFM}:
    
\begin{defi}[stable integral simplicial volume]
Let $M$ be an oriented closed connected $n$-manifold. 
The \emph{stable integral simplicial volume of $M$} is defined as
$$
\stisv M := \inf\Bigl\{\frac{\|N\|_{\Z}}{d}
\Bigm| d \in \N,\ N \mbox{ is a $d$-sheeted covering space of } M
\Bigr\}.
$$
\end{defi}

Similarly to Question~\ref{quest:ifsv:vanishes}, working with stable
integral simplicial volume the leading question is the following:

\begin{question}\label{quest:IA}
Let $M$ be an oriented closed connected aspherical $n$-manifold with 
$\|M\| = 0$ and residually finite fundamental group.
Then, do we have $\stisv M = 0$?
\end{question}

The question is known to have a positive answer in the 
the case of aspherical surfaces~\cites{vbc}, aspherical $3$-manifolds~\cites{FFL, FLMQ},
smooth aspherical manifolds admitting a smooth $S^1$-action without fixed points~\cite{FauserS1}, 
smooth aspherical manifolds admitting a regular smooth circle foliation with additional properties~\cite{Camp-Corro}, generalised graph manifolds~\cite{FFL}, and closed aspherical manifolds with vanishing minimal volume~\cite{braunsauer}. 

%%%%%%%%
\subsection{Integral foliated simplicial volume}

Integral foliated simplicial volume is defined via simplicial
volumes with twisted coefficients in~$\linfz X$. The basic
idea is that $\linfz X$ provides enough rigidity through the
constraint that functions are integer-valued, but also enough
flexibility through partitions of~$X$. We recall the
terminology in more detail:

Let $\Gamma$ be a countable group. A \emph{standard $\Gamma$-space}
is a measure preserving action~$\Gamma \actson (X,\mu)$ of~$\Gamma$
on a standard Borel probability space~$(X,\mu)$. We then equip~$\linfz X$
with the $\Z\Gamma$-module structure given by
\[ \gamma \cdot f := \bigl(x \mapsto f(\gamma^{-1}\cdot x)\bigr)
\]
for all~$\gamma \in \Gamma$ and all~$f \in \linfz X$. We 
consider homology of spaces with fundamental group~$\Gamma$
with twisted coefficients in~$\linfz X$.

\begin{defi}
  Let $M$ be a connected CW-complex (or a path-connected space
  that admits a universal covering) with countable fundamental
  group~$\Gamma$ and let $\alpha \colon \Gamma \actson(X,\mu)$ be a
  standard $\Gamma$-space.

  Let $n\in \N$ and let $x \in H_n(M;\Z)$. We write
  $x^\alpha \in H_n\bigl(M; \linfz X\bigr)
  $
  for the image of~$x$ under the change of coefficients
  map~$\Z \hookrightarrow \linfz X$.
  %Here, $\linfz X$
  %carries the $\Z\Gamma$-module structure induced by
  %the $\Gamma$-action on~$X$ and $\Z \hookrightarrow \linfz X$
  %is the inclusion as constant functions.

  We define the \emph{$\alpha$-parametrised norm of~$x$} by
  \[ \ifsvp x \alpha
  := \inf \bigl\{ |z|_1 \bigm| z \in \linfz X \otimes_{\Z \Gamma} C_*(\widetilde M;\Z)
  \text{ is a cycle representing~$x^\alpha$}\bigr\},
  \]  
  where
  \[ \biggl| \sum_{i \in I} f_i \otimes \sigma_i \biggr|_1
  := \sum_{i \in I} \int_X |f_i|\;d\mu
  \]
  for all chains~$\sum_{i \in I} f_i \otimes \sigma_i \in \linfz X
  \otimes_{\Z\Gamma} C_*(\widetilde M;\Z)$ in reduced form.
\end{defi}

\begin{defi}[integral foliated simplicial volume~\cites{gromov_metric, mschmidt}]
  Let $M$ be an oriented closed connected $n$-manifold with
  fundamental group~$\Gamma$.
  \begin{itemize}
  \item If $\alpha \colon \Gamma \actson (X,\mu)$ is a standard
    $\Gamma$-space, then the \emph{$\alpha$-parametrised simplicial
    volume} of~$M$ is defined as
    \[ \ifsvp M \alpha := \ifsvp[big]{[M]} \alpha.
    \]
  \item The \emph{integral foliated simplicial volume~$\ifsv M$ of~$M$}
    is defined as the infimum of~$\ifsvp M \alpha$ over all standard
    $\Gamma$-spaces~$\alpha$.
  \end{itemize}
\end{defi}

For further background on the integral foliated simplicial volume,
we refer to the literature~\cites{mschmidt, loehergodic}.

%%%%%%%%
\subsection{Amenable open covers}

We recall amenable open covers; more information on amenable groups
and their dynamical properties can be found in the
literature. 

\begin{defi}[amenable subset, amenable open cover]
  Let $M$ be a topological space.
  \begin{itemize}
  \item A subset~$U \subset M$ is called \emph{amenable} if the
    subgroup~$\im (\pi_1(U,x) \to \pi_1(M,x))$ of~$\pi_1(M,x)$ is
    amenable for every~$x \in U$.
  \item An open cover of~$M$ is called \emph{amenable} if each
    member is amenable in the sense above.
  \end{itemize}
\end{defi}

%%%%%%%
\subsection{Resolutions from free actions}

We fix basic notation concerning resolutions and recall the fundamental theorem of homological algebra. 

\begin{defi}\label{def:extends}
  Let $\Gamma$ be a group, let $A$ and $B$ be $\Z\Gamma$-modules,
  and let $\varepsilon \colon C_* \to A$ and $\eta \colon D_* \to B$
  be augmented $\Z\Gamma$-chain complexes. A $\Z\Gamma$-chain
  map~$\varphi \colon C_* \to D_*$ \emph{extends} a $\Z\Gamma$-homomorphism~$f \colon A \to B$
  if
  \[ \eta \circ \varphi_0 = f \circ \varepsilon.
  \]
\end{defi}

\begin{lem}[Fundamental theorem of homological algebra]\label{lem:extendsunique}
  Let $\Gamma$ be a group, let $B$ be a $\Z\Gamma$-module, let
  $\varepsilon \colon C_* \to \Z$ be a projective
  $\Z\Gamma$-resolution of~$\Z$, and let $\eta \colon D_* \to B$ be a
  $\Z\Gamma$-resolution of~$B$ (not necessarily projective). 
  If $\varphi, \varphi' \colon C_* \to D_*$ extend the
  same $\Z\Gamma$-homomorphism~$\Z \to B$, then we have $\varphi
  \simeq_{\Z\Gamma} \varphi'$.
\end{lem}
% \begin{proof}
%   This is the fundamental theorem of homological
%   algebra.
% \end{proof}

\begin{example}\label{exa:actionres}
  Let $\Gamma$ be a group and let $\Gamma \actson E$ be a free
  $\Gamma$-action on a set~$E$. For each~$n \in \N$, the product~$E^{n+1}$
  carries the diagonal $\Gamma$-action. Then $\Z[E^{*+1}]$ with the simplicial
  boundary operator given by
  \[ \partial_n(e_0,\dots, e_n)
  := \sum_{j=0}^n (-1)^j \cdot (e_0,\dots, \widehat e_j, \dots, e_n)
  \]
  for all~$n \in \N$, $(e_0, \dots, e_n) \in E^{n+1}$, together with the
  augmentation
  \begin{align*}
    \varepsilon_E \colon \Z[E^{0+1}] & \longrightarrow \Z
    \\
    E \ni e_0 & \longmapsto 1,
  \end{align*}
  is a free $\Z\Gamma$-resolution of the trivial $\Z\Gamma$-module~$\Z$.
\end{example}

\begin{example}\label{exa:asphres}
  Let $M$ be a path-connected aspherical space with fundamental
  group~$\Gamma$.  Then the singular chain complex~$C_*(\widetilde
  M;\Z)$ of the universal covering~$\widetilde M$ of~$M$, together
  with the augmentation~$C_0(\widetilde M;\Z) \to \Z$ mapping singular
  $0$-simplices to~$1$, is a free $\Z\Gamma$-resolution of the trivial
  $\Z\Gamma$-module~$\Z$.
\end{example}

\begin{defi}[essential chains, essential norm] 
  Let $\Gamma$ be a group, let $E$ be a free $\Gamma$-set, 
  and let $n \in \N$.
  \begin{itemize}
  \item A tuple~$(x_0,\dots, x_n) \in E^{n+1}$ is \emph{degenerate}
    if there are~$k,l \in \{0,\dots,n\}$ with~$k \neq l$ and~$x_k = x_l$.
    Non-degenerate tuples are called \emph{essential}.
  \item  
    We write~$\esst E {n+1}$ for the set of all essential tuples in~$E^{n+1}$.
  \item
    For a normed $\Z$-module~$A$ and a reduced chain
    $c = \sum_{x \in E^{n+1}} a_x \otimes x \in A \otimes_{\Z} \Z[E^{n+1}]$,
    we set
    \begin{align*}
      \essn c
      & := \sum_{x \in \esst E {n+1}} |a_x|,
      \\
      |c|_1
      & := \sum_{x \in E^{n+1}} |a_x|.
    \end{align*}
  \end{itemize}
\end{defi}

Note that $\essn \args$ is a semi-norm on $A \otimes_{\Z} \Z[E^{n+1}]$ whose restriction to essential chains is a norm.

%%%%%%%%%%%%%%%%%%%%%%%%%%%%%%%%%%%%%%%%%%%%%%%%%%%%%%%%%%%%%%%
\section{Outline of the proof of Theorem~\ref{thm:main}}\label{sec:outline}

The proof of Theorem~\ref{thm:main} is a dynamical version of
the amenable reduction lemma~\cites{gromov-cycles,alpert, alpertkatz}. 
As in Theorem~\ref{thm:main}, let $M$ be an oriented closed connected
smooth aspherical manifold with~$\amcat M \leq \dim M$ and let $\alpha
\colon \pi_1(M) \actson (X,\mu)$ be an essentially free standard
$\pi_1(M)$-space.
Let $(U_i)_{i \in I}$ be a finite amenable open cover of~$M$
with multiplicity at most~$n := \dim M$.  Let $(\Gamma_i)_{i \in I}$
be the (amenable) subgroups of~$\Gamma := \pi_1(M)$ associated
with~$(U_i)_{i \in I}$.

To prove that $\ifsvp M \alpha = 0$, we proceed in the following
steps:  Let $E := \Gamma \times I$, with
the $\Gamma$-translation action on the first factor.  Because $M$ is
aspherical, we can replace the $\Z\Gamma$-chain
complex~$C_*(\widetilde M;\Z)$ by~$\Z[E^{*+1}]$.
\begin{itemize}
\item 
  Refining the open cover~$(U_i)_{i \in I}$ and using a subdivision
  process, we obtain a fundamental cycle
  \[ z' = \sum_{\sigma \in E^{n+1}} a_\sigma \otimes \sigma \in \Z \otimes_{\Z \Gamma} \Z[E^{n+1}]
  \]
  of~$M$ with the following property:
  If $\sigma = ((\gamma_0, i_0), \dots, (\gamma_n, i_n))$
  is a simplex with~$a_\sigma \neq 0$,
  then $\sigma$ satisfies the ``colouring condition'' 
  \[ \exi{j,k \in \{0,\dots,n\}} j \neq k
  \qand
  i_j = i_k
  \qand 
  \gamma_j^{-1} \cdot \gamma_k \in \Gamma_{i_j} = \Gamma_{i_k}.
  \]
  This is a pigeon-hole principle argument based on~$\mult
  U \leq n$.
\item
  In the complex~$\Z[E^{*+1}]$, we
  have an additional degree of freedom. Using Rokhlin partitions
  of~$X$ coming from the amenable groups~$\Gamma_i$, we find good
  equivariant Borel partitions of the $\Gamma$-space~$X \times E$;
  such partitions lead to the Rokhlin chain map~$\Z[E^{*+1}] \to \linfz X
  \otimes_{\Z} \Z[\dots^{*+1}]$. We apply these Rokhlin chain
  maps to~$z'$.
  Because $z'$ satisfies the colouring condition, each simplex
  in~$z'$ has one edge that is susceptible to amenability of
  the~$\Gamma_i$; thus, 
  for each~$\delta \in \R_{>0}$ we
  find a fundamental cycle~$z'_\delta$ that admits a
  decomposition
  \begin{align*}
    z'_\delta
    =
    & \;\text{a chain consisting of degenerate simplices}
    \\
    +
    & \;\text{a chain of norm in~$o(\delta)$}.
  \end{align*}
  Here, the ``almost invariance'' property of the amenable
  groups~$\Gamma_i$ is reflected in the fact that almost everything
  cancels, except for degenerate terms. 
\end{itemize}
As degenerate simplices can be neglected when considering
the $\ell^1$-semi-norm on homology, taking~$\delta \to 0$
shows that $\ifsvp M \alpha = 0$. This strategy can be subsumed
in the following diagram:
\tikzset{away/.style={outer sep=.5em}}
\[ \begin{tikzcd}
  C_*(\widetilde M;\Z)
  \arrow[r,away,"\text{equivariant subdivision}"]
  \arrow[d,dashed,away,"\text{$\simeq$ incl}"']
  & \Z[E^{*+1}]
  \ar[d,away,"\text{Rokhlin}"]
  \\
   \linfz X \otimes_\Z C_*(\widetilde M;\Z)
  & \linfz X \otimes_\Z \Z[\ldots^{*+1}]
   \ar[l,away,"\text{filling}"]
   \ar[loop right,dashed,"\text{essentialness}"]
  \end{tikzcd}
\]

Section~\ref{sec:rokhlin} is devoted to the Rokhlin map;
Section~\ref{sec:subdiv} contains the subdivision argument. 
The actual proof of Theorem~\ref{thm:main} is given in
Section~\ref{sec:proof}. 

%%%%%%%%%%%%%%%%%%%%%%%%%%%%%%%%%%%%%%%%%%%%%%%%%%%%%%%%%%%%%%%
\section{The Rokhlin chain map}\label{sec:rokhlin}

We construct the Rokhlin chain map, we determine its
effect on homology, and -- most importantly -- establish
the key norm estimates.

We first explain how Borel partitions lead to chain maps.
The Rokhlin chain maps then are the chain maps associated
with Rokhlin partitions.

%%%%%%%%%%%%%
\subsection{Borel chains}

Let $\Gamma \actson (X,\mu)$ be a standard $\Gamma$-space and let $S$
be a countable free $\Gamma$-set. We endow~$X \times S^{n+1}$ with the
diagonal $\Gamma$-action and the product measure of~$\mu$ and the
counting measure on~$S$.
We write
\[ \lfsz{X \times S^{n+1}}
\]
for the submodule of~$\linfz{X \times S^{n+1}}$ consisting of
those (equivalence classes of) essentially bounded
functions~$X \times S^{n+1} \to \Z$ whose support is
contained in a set of the form~$X \times F$, where
$F \subset S^{n+1}$ is finite.
For~$n \in \N_{>0}$, we let
\begin{align*}
  \partial_n \colon \lfsz {X \times S^{n+1}}
  & \longrightarrow \lfsz {X \times S^n}
  \\
  f
  & \longmapsto
  \biggl(
  (x,s)
  \mapsto
  \sum_{j=0}^n (-1)^j
  \cdot \sum_{t \in S}
  f(x, s_0,\dots, s_{j-1}, t, s_j, \dots, s_{n-1})
  \biggr).
\end{align*}
This turns~$\lfsz{X \times S^{*+1}}$ into a $\Z\Gamma$-chain
complex. 
We consider the augmentation map
\begin{align*}
  \widetilde \varepsilon_S
  \colon \lfsz {X \times S^{0+1}}
  & \longrightarrow \linfz X
  \\
  f
  & \longmapsto \sum_{s \in S} f(\args,s).
\end{align*}
Furthermore, we equip~$\lfsz {X \times S^{*+1}}$ with the
\emph{essential norm} given by
\[ \essn f := \sum_{s \in \esst S {n+1}} \int_X |f(\args,s)| \;d\mu(x)
\]
for all~$f \in \lfsz {X \times S^{n+1}}$.

\begin{lem}\label{lem:lfsz}
  Let $\Gamma$ be a countable group, let $S$ be a free $\Gamma$-set,
  and let $\Gamma \actson (X,\mu)$ be a 
  standard $\Gamma$-space.
  Then
  \begin{align*}
    \xi_S \colon 
    \linfz X \otimes_\Z \Z[S^{*+1}]
    & \longrightarrow \lfsz{X \times S^{*+1}}
    \\
    f \otimes s
    & \longmapsto
    \left(
    (x,t) \mapsto
    \begin{cases}
      f(x) & \text{if $t = s$}
      \\
      0 & \text{if $t \neq s$}
    \end{cases}
    \right)
  \end{align*}
  and
  \begin{align*}
    \zeta_S \colon 
    \lfsz{X \times S^{*+1}}
    & \longrightarrow \linfz X \otimes_\Z \Z[S^{*+1}]
    \\
    f
    & \longmapsto
    \sum_{s \in S^{n+1}}
    f(\args,s) \otimes s
  \end{align*}
  are well-defined mutually inverse $\Z\Gamma$-chain isomorphisms,
  where the left hand side is endowed with the diagonal $\Gamma$-action. 

  These chain maps are compatible with the
  augmentations~$\id_{\linfz X} \otimes_\Z \varepsilon_S$ and
  $\widetilde \varepsilon_S$ as well as with the essential norms~$\essn{\args}$.  
\end{lem}
\begin{proof}
  This is a straightforward computation.
\end{proof}

\begin{cor}\label{cor:lfszres}
  In the situation of Lemma~\ref{lem:lfsz},
  $\widetilde\varepsilon_S \colon \lfsz{X \times {S^{*+1}}} \to \linfz X$
  is a $\Z\Gamma$-resolution of~$\linfz X$.
\end{cor}
\begin{proof}
  The chain complex $\Z[S^{*+1}]$ with augmentation $\varepsilon_S \colon \Z[S^{*+1}] \to \Z$ is 
  a $\Z\Gamma$-resolution of~$\Z$. The $\Z$-module $\linfz X$ is torsionfree, thus flat. So 
  $\linfz X \otimes_\Z \Z[S^{*+1}]\to \linfz X$ is exact. With the diagonal action on 
  $\linfz X \otimes_\Z \Z[S^{*+1}]$ it becomes a $\Z[\Gamma]$-resolution of $\linfz X$. 
  Lemma~\ref{lem:lfsz} concludes the proof. 
\end{proof}

%%%%%%%%%%%%%
\subsection{The chain map associated with a Borel partition}

Borel partitions lead to corresponding equivariant partition
chain maps.

\begin{defi}\label{def:equivcover}
  Let $\Gamma$ be a countable group, let $E$ be a countable free $\Gamma$-set,  
  and let $\Gamma \actson (X,\mu)$ be a free standard
  $\Gamma$-space.
  A \emph{$\Gamma$-equivariant $\mu$-partition of~$X \times E$} is
  a family~$P:= (W_s)_{s \in S}$ of Borel subsets of~$X \times E$ with
  the following properties:
  \begin{itemize}
  \item The sets~$(W_s)_{s \in S}$ are pairwise disjoint. 
  \item The union~$\bigcup_{s \in S} W_s$ is $\mu\otimes\delta$-conull
    in~$X \times E$, where $\delta$ denotes the counting measure on~$E$.
  \item The set~$S$ is equipped with a free $\Gamma$-action and for
    all~$\gamma \in \Gamma$, $s\in S$, we have
    \[ \gamma \cdot W_s = W_{\gamma \cdot s}.
    \]
  \end{itemize}
  We say that $P$ is \emph{of finite type} if the following holds:
  For each~$e \in E$, there are only finitely many~$s \in S$
  with~$W_s \cap X \times \{e\} \neq \emptyset$.
\end{defi}
  
\begin{rem}[induced partitions]\label{rem:fibreproduct}
  Let $\Gamma$ be a countable group, let $E$ be a countable
  free $\Gamma$-set, let $\Gamma \actson (X,\mu)$ be a 
  free standard $\Gamma$-space, and let $P = (W_s)_{s \in S}$ be a $\Gamma$-equivariant
  $\mu$-partition of~$X \times E$. 
  
  If $n \in \N$ and $s_0,\dots, s_n \in S$, we write
  \begin{align*}
    W_{(s_0,\dots, s_n)}
    & := W_{s_0} \times_X \dots \times_X W_{s_n}
    \\
    & := \bigl\{ (x, e_0, \dots, e_n)
    \bigm| \fa{r \in \{0,\dots,n\}} (x,e_r) \in W_{s_r}
    \bigr\}
    \subset X \times E^{n+1}.
  \end{align*}

  Then, for each~$n \in
  \N$, the family~$(W_s)_{s \in S^{n+1}}$ is a $\Gamma$-equivariant
  $\mu$-partition of~$X \times E^{n+1}$. If $P$ is of finite type,
  then also $(W_s)_{s \in S^{n+1}}$ is so.
\end{rem}

\begin{prop}\label{prop:borelpchmap}
  Let $\Gamma$ be a countable group, let $E$ be a free $\Gamma$-space,
  let $\Gamma \actson (X,\mu)$ be a free standard
  $\Gamma$-space, and let $P:= (W_s)_{s \in S}$ be a $\Gamma$-equivariant
  $\mu$-partition of~$X \times E$ of finite type.
  Then
  \begin{align*}
    \Phi^P \colon
    \Z[E^{*+1}]
    & \longrightarrow \lfsz{X \times S^{*+1}}
    \\
    (e_0,\dots,e_n)
    & \longmapsto
    \bigl((x,s_0,\dots,s_n) \mapsto \chi_{W_{(s_0,\dots,s_n)}}(x,e_0,\dots,e_n) \bigr) 
  \end{align*}
  is a $\Z\Gamma$-chain map that extends $\Z \hookrightarrow \linfz X$.
\end{prop}
\begin{proof}
  Let $n \in \N$, let $e \in E^{n+1}$, and let~$f := \Phi^P(e)$.

  We first show that $f$ indeed lies in~$\lfsz{X \times S^{*+1}}$.
  Clearly, $f$ is measurable and essentially bounded by~$1$
  because $(W_s)_{s \in S^{n+1}}$ is a $\mu$-partition of~$X \times E^{n+1}$
  (Remark~\ref{rem:fibreproduct}). Moreover, $f$ has
  finite support as $P$ is of finite type.

  For $\Gamma$-equivariance, we compute for all~$(x,s) \in X \times
  S^{n+1}$ that (using that $P$ is $\Gamma$-equivariant)
  \begin{align*}
    \bigl(\gamma \cdot (\Phi^P(e))\bigr)(x,s)
%    & = f(\gamma^{-1}\cdot x, \gamma^{-1} \cdot s_0, \dots, \gamma^{-1}\cdot s_n)
%    \\
    &= \chi_{W_{\gamma^{-1}\cdot s}}(\gamma^{-1}\cdot x, e)
    \\
    & = \chi_{\gamma^{-1} \cdot W_s}(\gamma^{-1}\cdot x, e)
    \\
    & = \chi_{W_s}(x,\gamma \cdot e)
    \\
    & = \bigl(\Phi^P(\gamma\cdot e)\bigr)(x,s).
  \end{align*}

  For the chain map property, we compute for all~$j \in \{0,\dots,n\}$
  and all~$s \in S^n$ that
  (using that $P$ induces a partition on~$X \times E^{n+1}$)
  \begin{align*}
    \sum_{t \in S} \Phi^P(e) (x,s_0,\dots,s_{j-1},t,s_j, \dots, s_{n-1})
    & = \sum_{t \in S} \chi_{W_{(s_0,\dots, s_{j-1},t,s_j, \dots, s_{n-1})}}(x,e)
    \\
    & = \chi_{\bigcup_{t \in S} W_{(s_0,\dots, s_{j-1},t,s_j, \dots, s_{n-1})}}(x,e)
    \\
    & = \chi_{W_{(s_0,\dots,s_{n-1})}}(x,e_0,\dots,\widehat e_j, \dots,e_n).
  \end{align*}
  Therefore, $\partial \bigl(\Phi^P(e)) = \Phi^P(\partial e)$.

  Finally, we consider degree~$0$: 
  As $\widetilde \varepsilon_S \colon \lfsz {X \times S^{*+1}} \to \linfz X$
  is a $\Z\Gamma$-resolution (Corollary~\ref{cor:lfszres}), the notion
  of ``extends'' from Definition~\ref{def:extends} is available. 
  To see that the extension condition is satisfied, it suffices to 
  show that~$\widetilde \varepsilon_S \circ \Phi^P(e_0)$ is the constant
  function~$1$ 
  for each~$e_0 \in E^{0+1}$: By construction,
  \begin{align*}
    \widetilde \varepsilon_S \circ \Phi^P(e_0)
    & =
    \sum_{s \in S} \bigl(x \mapsto \chi_{W_s}(x,e_0)\bigr).
  \end{align*}
  Because $(W_s)_{s\in S}$ is a $\mu$-partition of~$X \times E$,
  this sum equals~$1$.
\end{proof}

%%%%%%%%%%%%%
\subsection{The Rokhlin chain map}

The Rokhlin chain maps are chain maps obtained by applying the
construction of the previous section to Rokhlin partitions.

We use the following version of the Ornstein--Weiss Rokhlin lemma,
which allows us to perform divisions on the level of Borel sets:

\begin{thm}[Rokhlin lemma~\protect{\cite[Theorem~3.6]{rokhlin}}]
  \label{thm:rokhlinlemma}
  Let $\Gamma$ be a countable amenable group, let $\Gamma \actson
  (X,\mu)$ be a free standard $\Gamma$-action,
  let $F \subset \Gamma$ be a finite set, and let $\delta \in
  \R_{>0}$.

  Then, there exists a $\mu$-conull $\Gamma$-invariant Borel set~$X' \subset X$,
  a finite set~$J$, a family~$(A_j)_{j \in J}$ of Borel subsets of~$X'$,
  and a family~$(T_j)_{j \in J}$ of $(F,\delta)$-invariant non-empty finite subsets
  of~$\Gamma$ such that $(T_j \cdot x)_{j \in J, x \in A_j}$
  partitions~$X'$.
\end{thm}

Here, a non-empty finite subset~$T \subset \Gamma$ is
\emph{$(F,\delta)$-invariant} if
\[ \frac{|F \cdot T \symmdiff T|}{|T|} \leq \delta.
\]

\begin{setup}\label{setup:rokhlin}
  Let $\Gamma$ be a countable group, let $(\Gamma_i)_{i \in I}$ be a
  non-empty finite family of amenable subgroups of~$\Gamma$, and let $F
  \subset \Gamma$ be a finite set.

  Moreover, let $\Gamma \actson (X,\mu)$ be an essentially free 
  standard $\Gamma$-space.

  Let $\delta \in \R_{>0}$; for each~$i \in I$, let $((A_{i,j})_{j \in
    J_i}, (T_{i,j})_{j \in J_i})$ be a $(F\cap \Gamma_i,\delta)$-Rokhlin
  partition of a $\Gamma_i$-invariant $\mu$-conull subset~$X'_i$ of~$X$
  with respect to~$F \cap \Gamma_i$, as provided by
  Theorem~\ref{thm:rokhlinlemma}. Without loss of generality 
  we also assume that each~$T_{i, j}$ contains the neutral element of~$\Gamma$.

  Let $E := \Gamma \times I$, with the $\Gamma$-translation action
  on the first factor.
\end{setup}

\begin{rem}\label{rem:rokhlinmeasure}
  For later use, we record that in the situation of Setup~\ref{setup:rokhlin},
  we have
  \[ \mu(A_{i,j}) = \frac{\mu(T_{i,j} \cdot A_{i,j})}{|T_{i,j}|}
  \]
  for all~$i \in I$ and all~$j \in J_i$, by the partition condition.
  
  Moreover, given~$i \in I$ and $j_0, j_1 \in J_i$ with~$j_0 \neq j_1$,
  we have~$A_{i,j_0} \cap A_{i,j_1} = \emptyset$ because $T_{i,j_0}$
  and $T_{i,j_1}$ contain the neutral element.
\end{rem}
  
\begin{defi}[the Rokhlin chain map]\label{rem:equivrokhlin}
  In the situation of Setup~\ref{setup:rokhlin}, the Rokhlin partition
  gives a $\Gamma$-equivariant $\mu$-partition of~$X \times \Gamma \times I$ which has finite type,
  namely
  \[ P^\delta := 
    \bigl((\gamma \cdot A_{i,j})
    \times (\gamma \cdot T_{i,j}{}^{-1})
    \times \{i\}\bigr)_{\gamma \in \Gamma, i \in I, j \in J_i}. 
  \]
  We write~$S := \{ (\gamma,i,j) \bigm| \gamma \in \Gamma, i \in I, j
  \in J_i \}$, equipped with the $\Gamma$-translation action on the
  first factor. For~$s = (\gamma,i,j) \in S$, we set
  \[ W_s := (\gamma \cdot A_{i,j}) \times (\gamma \cdot T_{i,j}^{-1}) \times \{i\}.
  \]
  Finally, we abbreviate~$\Phi^\delta := \Phi^{P^\delta} \colon
  \Z[E^{*+1}] \longrightarrow \lfsz {X \times S^{*+1}}$
  (Proposition~\ref{prop:borelpchmap}). 
\end{defi}

\begin{defi}[colouring condition]
  In the situation of Setup~\ref{setup:rokhlin}, we say that
  a tuple~$z = (\lambda_r,t_r)_{r \in [n]} \in E^{n+1}$ satisfies
  the \emph{$F$-colouring condition} if there exists an~$i \in I$
  and $k,l \in [n]$ with~$k < l$ such that $t_k = t_l = i$ and
  \[ \lambda_k^{-1} \cdot \lambda_l \in F \cap \Gamma_i.
  \]
  Here, $[n] := \{0,\dots,n\}$.
\end{defi}

In the situation of Setup~\ref{setup:rokhlin}, we always consider $S$
to be the set described in Definition~\ref{rem:equivrokhlin}.

\begin{thm}[properties of the Rokhlin chain map]\label{thm:rokhlinmap}
  In the situation of Setup~\ref{setup:rokhlin}, 
  the Rokhlin chain map~$\Phi^\delta \colon \Z[E^{*+1}] \longrightarrow
  \lfsz {X \times S^{*+1}}$ has the following properties: 
  \begin{enumerate}
  \item The map~$\Phi^\delta$ is a $\Z\Gamma$-chain map that
    extends~$\Z \hookrightarrow \linfz X$.
  \item Let $z \in \Z[E^{*+1}]$ be a chain that satisfies the
    $F$-colouring condition. Then $\essn {\Phi^\delta(z)} \leq \delta \cdot |z|_1$;
    in particular, 
    \[ \lim_{\delta \to 0} \;
       \essn[big] {\Phi^\delta(z)} = 0.
    \]
  \end{enumerate}
\end{thm}

\begin{proof}
  \emph{Ad~1.} This is a general property for chain maps
  associated with Borel partitions (Proposition~\ref{prop:borelpchmap}).

  \emph{Ad~2.} The proof is given in Section~\ref{subsec:rokhlinnorm}.
\end{proof}

%%%%%%%%%%%%%
\subsection{The norm estimate for the Rokhlin chain map}\label{subsec:rokhlinnorm}

We now prove the second part of Theorem~\ref{thm:rokhlinmap}.
In view of the triangle inequality, it suffices to prove that
\[ 
   \essn[big] {\Phi^\delta(z)}\leq \delta 
\]
holds for all tuples~$z \in E^{*+1}$ that satisfy the $F$-colouring
condition. Since all the involved terms are sufficiently $\Gamma$-equivariant
and symmetric in the coordinates of~$z = (\lambda_r,t_r)_{r \in [n]}$,
we may assume without loss of generality that~$\lambda_0 = 1$ and
$\lambda_1 \in F \cap \Gamma_i$, where $i = t_0 = t_1$.

\begin{setup}\label{setup:rokhlinnorm}
  We assume Setup~\ref{setup:rokhlin}. In addition, we let $n \in \N$
  and we let 
  \[ z = \bigl( (\lambda_0, t_0), \dots, (\lambda_n,t_n)\bigr) \in E^{n+1}  
  \]
  be a tuple that satisfies the $F$-colouring condition in the first
  two coordinates. More specifically,
  let $\lambda_0  =1$, and let $i \in I$ with $i = t_0 = t_1$ and 
  \[ \lambda_1 \in F \cap \Gamma_i.
  \]
  Furthermore, we set
  \[ \esst{S_z}{n+1}
  :=
  \bigl\{ (\gamma_r,i_r,j_r)_{r \in [n]} \in \esst S {n+1}
  \bigm|
  \fa{r \in [n]} i_r = t_r,\ 
  \fa{r \in [n]} \gamma_r \in \lambda_r \cdot T_{t_r,j_r}
  \bigr\}.
  \]
\end{setup}

By definition, 
\[ \essn[big]{\Phi^\delta(z)}
\leq \sum_{s \in \esst S {n+1}} \int_X \chi_{W_s}(x,z) \; d\mu(x).
\]

As a first reduction, we show that only summands in~$\esst{S_z}{n+1}$
contribute to this sum:

\begin{lem}\label{lem:Sz}
  In the situation of Setup~\ref{setup:rokhlinnorm}, 
  let $s \in \esst S{n+1} \setminus \esst{S_z}{n+1}$.
  Then
  \[ \int_X \chi_{W_s}(x,z) \; d\mu(x) = 0.
  \]
\end{lem}
\begin{proof}
  We write~$s = (\gamma_r,i_r,j_r)_{r \in [n]}$
  and compute 
  \begin{align*}
    \int_X \chi_{W_s}(x,z) \; d\mu(x)
    & =
    \mu\bigl(\bigl\{ x \in X
    \bigm| (x,z) \in W_s
    \bigr\}\bigr)
    \\
    & = 
    \mu\bigl(\bigl\{ x \in X
    \bigm| \fa{r\in[n]} (x,\lambda_r,t_r) \in W_{\gamma_r,i_r,j_r}
    \bigr\}\bigr).
  \end{align*}
  As $s \not\in \esst{S_z}{n+1}$, we can distinguish two cases:
  \begin{itemize}
  \item 
    there exists an~$\overline{r} \in [n]$
    with $i_{\overline{r}} \neq t_{\overline{r}}$;
  \item
    for all~$r \in [n]$, we have~$i_r = t_r$ but 
    there exists an~$R \in [n]$
    with~$\gamma_{R} \not \in \lambda_{R} \cdot T_{t_{R},j_{R}}$.
  \end{itemize}
  
  In the first case, $(x,\lambda_{\overline{r}},t_{\overline{r}}) \not\in W_{\gamma_{\overline{r}},i_{\overline{r}},j_{\overline{r}}}$
  in view of the $I$-coordinate,
  and so the set in question is empty.

  In the second case, we obtain
  \begin{align*}
    \int_X \chi_{W_s}(x,z) \; d\mu(x)
    & =
    \mu\bigl(\bigl\{ x \in X
    \bigm|
    \fa{r \in [n]} (x,\lambda_r) \in \gamma_r \cdot A_{t_r,j_r} \times \gamma_r \cdot T^{-1}_{t_r,j_r}
    \bigr\}\bigr);
  \end{align*}
  this latter set is empty because of the $R$-th term.

  Thus, in both cases, the integral is zero.
\end{proof}

\begin{lem}\label{lem:lastestimate}
  In the situation of Setup~\ref{setup:rokhlinnorm}, we have
  \[ \essn[big]{\Phi^{\delta}(z)} \leq \delta.
  \]
\end{lem}
\begin{proof}
  By the reduction in Lemma~\ref{lem:Sz} and the definition of~$\esst{S_z}{n+1}$
  and~$W_s$, we have
  \begin{align*}
    \essn[big]{\Phi^{\delta}(z)}
    & \leq \sum_{s \in \esst{S_z}{n+1}} \int_X \chi_{W_s}(x,z) \;d\mu(x)
    \\
    & = \sum_{(\gamma_r,i_r,j_r)_r \in \esst{S_z}{n+1}}
    \mu \Bigl(\bigcap_{r \in [n]} \gamma_r \cdot A_{t_r,j_r} \Bigr)
    \\
    & = \sum_{(\gamma_r,i_r,j_r)_r \in \esst{S_z}{n+1}}
    \mu \Bigl(\gamma_0 \cdot A_{i,j_0} \cap \gamma_1 \cdot A_{i,j_1}
    \cap \bigcap_{r=2}^n \gamma_r \cdot A_{t_r,j_r} \Bigr). 
  \end{align*}
  By the Rokhlin partition condition, for each~$r \in [n]$, the
  sets in the family~$(\gamma_r \cdot A_{t_r,j_r})_{j_r \in J_{t_r}, \gamma_r \in \lambda_r T_{t_r,j_r}}$
  are pairwise disjoint. Therefore, inductively (over~$r \geq 2$), we can simplify to
  \begin{align*}
    \essn[big]{\Phi^{\delta}(z)}
    & \leq \sum_{j_0 \in J_i} \sum_{\gamma_0 \in T_{i,j_0}}
    \sum_{j_1 \in J_i} \sum_{\gamma_1 \in \lambda_1 \cdot T_{i,j_1} \setminus \{\gamma_0\}}
    \mu( \gamma_0 \cdot A_{i,j_0} \cap \gamma_1 \cdot A_{i,j_1})
    \\
    & \ +
    \sum_{j_0 \in J_i} \sum_{\gamma_0 \in T_{i,j_0}} \sum_{j_1 \in J_i \setminus\{j_0\}}
    \mu( \gamma_0 \cdot A_{i,j_0} \cap \gamma_0 \cdot A_{i,j_1})   
    .
  \end{align*}
  In the first sum, 
  if $\gamma_1 \in T_{i,j_1}$, then the intersection term is empty
  because we have a Rokhlin partition.
  Similarly, in the second sum, only empty sets appear (here we use the fact that
  in our setup $A_{i, j_0}$ and $A_{i, j_1}$ are disjoint; Remark~\ref{rem:rokhlinmeasure}).
  So, we obtain
  \begin{align*}
    \essn[big]{\Phi^{\delta}(z)}
    & \leq \sum_{j_0 \in J_i} \sum_{\gamma_0 \in T_{i,j_0}}
    \sum_{j_1 \in J_i} \sum_{\gamma_1 \in \lambda_1 \cdot T_{i,j_1} \setminus T_{i,j_1}}
    \mu( \gamma_0 \cdot A_{i,j_0} \cap \gamma_1 \cdot A_{i,j_1}).
  \end{align*}
  Again, using the Rokhlin partition, we see that the sets~$\gamma_0 \cdot A_{i,j_0}$
  are all pairwise disjoint for~$j_0 \in J_i$, $\gamma_0 \in T_{i,j_0}$. Thus,
  the sum reduces to 
  \begin{align*}
    \essn[big]{\Phi^{\delta}(z)}
    & \leq 
    \sum_{j_1 \in J_i} \sum_{\gamma_1 \in \lambda_1 \cdot T_{i,j_1} \setminus T_{i,j_1}}
    \mu(\gamma_1 \cdot A_{i,j_1})
    \\
    & =
    \sum_{j_1 \in J_i} \sum_{\gamma_1 \in \lambda_1 \cdot T_{i,j_1} \setminus T_{i,j_1}}
    \mu(A_{i,j_1}).
  \end{align*}
  Finally, we use the $(F\cap \Gamma_i,\delta)$-invariance of~$T_{i,j_1}$
  (and Remark~\ref{rem:rokhlinmeasure}):
  \begin{align*}
    \essn[big]{\Phi^{\delta}(z)}
    & \leq 
    \sum_{j_1 \in J_i}
    \mu(A_{i,j_1}) \cdot |\lambda_1 \cdot T_{i,j_1} \setminus T_{i,j_1}|
    \\
    & \leq \sum_{j_1 \in J_i}
    \frac{\mu(T_{i,j_1} \cdot A_{i,j_1})}{|T_{i,j_1}|}
    \cdot
    \delta \cdot |T_{i,j_1}|
    \\
    & = \delta \cdot \sum_{j_1 \in J_i} \mu(T_{i,j_1} \cdot A_{i,j_1})  
    \\
    & = \delta \cdot 1.
  \end{align*}
  This is the claimed estimate.
\end{proof}

This completes the proof of Theorem~\ref{thm:rokhlinmap}.

%%%%%%%%%%%%%%%%%%%%%%%%%%%%%%%%%%%%%%%%%%%%%%%%%%%%%%%%%%%%%%%%
\section{Fundamental cycles subordinate to open covers}\label{sec:subdiv}

To apply the Rokhlin chain map to obtain fundamental cycles of small
norm, we subdivide a fundamental cycle such that the resulting
cycle is subordinate to the given open cover. We formulate this
subdivision in the language of equivariant chain maps. 
Moreover, we show how these subdivisions can be integrated into
the norm estimates for simplicial volume.

\begin{setup}\label{setup:amenable:covers:top}
Let $n \in \, \N$. We consider the following situation
\begin{itemize}
\item Let $M$ be a path-connected finite CW-complex 
  with fundamental group $\Gamma := \pi_1(M)$; 
\item Let $\pi \colon \ucov M \to M$ be the universal covering
  map;
\item Let $(U_i)_{i \in \, I}$ be a finite open amenable cover of~$M$
  such that each~$U_i$ is path-connected.  For every~$i \in \, I$, let
  $\Gamma_i \leq \Gamma$ be~$\im(\pi_1(U_i,x_i) \to \pi_1(M,x_i))$
  for some choice of~$x_i \in U_i$ (up to conjugacy, this choice
  will have no effect).
\item Let $E := \Gamma \times I$, equipped with the $\Gamma$-action on
  the first factor.
\end{itemize}
\end{setup}

%%%%%%%
\subsection{From open covers to small equivariant open covers}

As a first step, we refine the given open cover and pass to
an equivariant setting.

\begin{defi}\label{def:small:equiv:cover}
  In the situation of Setup~\ref{setup:amenable:covers:top}, a \emph{small equivariant
    open cover associated with~$(U_i)_{i \in I}$} is an open
  cover~$(\gamma \cdot K_i)_{(\gamma,i) \in \Gamma \times I}$ of~$\ucov
  M$ with the following properties:
  \begin{itemize}
  \item For all~$i \in I$, the set~$K_i \subset \ucov M$ is relatively compact,
    open, and path-connected.
  \item For every~$i \in I$, the sets~$\gamma \cdot K_i$ with~$\gamma \in \Gamma$
    are pairwise distinct.
  \item For all~$i \in I$, the set~$\Gamma_i \cdot K_i$ is a path-component
    of~$\pi^{-1}(U_i)$.
  \item For all~$i \in I$ and all~$\gamma, \lambda \in \Gamma$,
    we have $\gamma \cdot K_i \cap \lambda \cdot K_i
    \neq \emptyset \Longrightarrow \gamma^{-1} \cdot \lambda \in
    \Gamma_i$.
  \end{itemize}
\end{defi}

\begin{lem}\label{lem:smallcovex}
  In the situation of Setup~\ref{setup:amenable:covers:top}, there exists
  a small equivariant open cover associated with~$(U_i)_{i\in I}$.
\end{lem}
\begin{proof}
  Let $i \in I$. We choose a path-connected component~$V_i$
  of~$\pi^{-1}(U_i)$; we make this choice in such a way
  that the deck transformation action of~$\Gamma$ on~$\widetilde M$
  restricts to an action of~$\Gamma_i$ to~$V_i$.
  The restriction~$\pi|_{V_i} \colon V_i \to U_i$
  is a covering map and we may choose a relatively compact, path-connected
  fundamental domain~$D_i \subset V_i$ for the deck transformation
  action~$\Gamma_i \actson V_i$. 
  We enlarge~$D_i$ to a relatively compact, open, path-connected
  subset~$K_i \subset V_i$ with $D_i \subset K_i$ and
  \[ \fa{\gamma \in \Gamma \setminus \{e\}} \gamma \cdot K_i \neq K_i.
  \]
  Such a set can, for instance, be obtained by selecting a point~$x_0
  \in D_i^\circ$ and then looking at relatively compact, open,
  path-connected neighbourhoods of all points in~$\overline{D_i}$ that
  avoid the discrete set~$\Gamma \cdot x_0 \setminus \{x_0\}$; as
  $D_i$ is relatively compact, finitely many such neighbourhoods
  suffice.
  
  By construction, $K_i$ is relatively compact, open, and
  path-connected.  Moreover, $\gamma \cdot K_i \neq \lambda \cdot K_i$
  for distinct group elements $\gamma, \lambda \in \Gamma$.

  We have~$\Gamma_i \cdot K_i = V_i$ because
  \[ V_i
  = \bigcup_{\gamma \in \Gamma_i} \gamma \cdot D_i
  \subset \bigcup_{\gamma \in \Gamma_i} \gamma \cdot K_i
  \subset \bigcup_{\gamma \in \Gamma_i} \gamma \cdot V_i
  = V_i.
  \]

  The last property is
  automatically satisfied because $\Gamma_i \cdot K_i$
  coincides with the path-component~$V_i$ of~$\pi^{-1}(U_i)$
  and $\Gamma_i$ is~$\im(\pi_1(U_i,x_i) \to \pi(M,x_i))$.
\end{proof}

%%%%%%
\subsection{Equivariant subdivision}\label{subsec:eqsubdiv}

In the situation of Setup~\ref{setup:amenable:covers:top}, let $\Gamma$
be torsion-free (e.g., this holds if $M$ is a closed aspherical manifold),
and we let $N$ denote the nerve of a small equivariant open cover~$(\gamma \cdot
K_i)_{(\gamma,i) \in \Gamma \times I}$ associated with~$(U_i)_{i \in
  I}$. Thus, $N$ is a $\Gamma$-simplicial complex and the obvious
simplicial $\Gamma$-action on~$N$ is free on the set of all simplices
(since $\Gamma$ is torsion-free). 

Moreover, we denote the geometric realisation of a simplicial
complex~$X$ by~$|X|$ and we write $C_*^s(X;\Z)$ for the ordered
simplicial chain complex~\cite[Chapter~4.3]{spanier}.

\begin{lem}\label{lem:nervemap}
  In this situation, there exists a continuous
  $\Gamma$-map~$\ucov M \longrightarrow |N|$.
\end{lem}
\begin{proof}
  Nerve maps associated with equivariant $\Gamma$-partitions
  of unity are continuous $\Gamma$-maps~\cite[proof of Lemma~4.8]{LS}.
  Therefore, it suffices to find a $\Gamma$-partition of unity
  of~$\ucov M$ subordinate to~$(\gamma \cdot K_i)_{(\gamma,i) \in \Gamma \times I}$.

  Starting with an open cover of~$\ucov M$ that refines~$(\gamma\cdot
  K_i)_{(\gamma,i) \in \Gamma \times I}$ and such that the universal covering map~$\pi$
  of~$M$ is a homeomorphism on each member, we find a finite open cover~$V$
  of~$M$ that refines~$(U_i)_{i \in I}$ with the following property:
  For each member~$W$ of~$V$, there exists an open subset~$\ucov W \subset \ucov M$
  such that $\pi|_{\ucov W} \colon \ucov W \to W$ is a homeomorphism and
  such that there exists an~$i(W) \in I$ with~$\ucov W \subset K_{i(W)}$.
  
  Let $(\varphi_W)_{W\in V}$ be a partition of unity of $M$ subordinate to $V$.
  Then, for~$(\gamma,i) \in \Gamma \times I$, we set
  \[ \widetilde \varphi_{(\gamma,i)}
  := \sum_{W \in V, i(W) = i}
  \chi_{\gamma \cdot \ucov W} \cdot \varphi_W \circ \pi.
  \]
  A straightforward computation shows that each~$\widetilde
  \varphi_{(\gamma,i)}$ is continuous and that
  $(\widetilde\varphi_{(\gamma,i)})_{(\gamma,i) \in \Gamma \times I}$ is a
  partition of unity that is subordinate to~$(\gamma \cdot
  K_i)_{(\gamma,i)\in \Gamma \times I}$. Moreover, we have
  \[ \fa{x \in \ucov M}
  \widetilde \varphi_{(\gamma,i)}(\lambda \cdot x)
  = \widetilde \varphi_{\lambda^{-1} \cdot (\gamma,i)}(x)
  \]
  for all~$(\gamma,i) \in \Gamma \times I$ and all~$\lambda \in \Gamma$.
  This shows that $(\widetilde \varphi_{(\gamma,i)})_{(\gamma,i) \in \Gamma \times I}$
  has the desired properties.
\end{proof}

\begin{lem}\label{lem:simplicialhomology}
  Let $\Gamma$ be a group and let $X$ be a simplicial complex
  with a simplicial $\Gamma$-action that is free on the set of
  all simplices.
  Then the canonical chain map~$C_*^s(X;\Z) \longrightarrow
  C_*(|X|;\Z)$~\cite[Chapter~4.4]{spanier} is a $\Z \Gamma$-chain
  homotopy equivalence.
\end{lem}
\begin{proof}
  The canonical chain map~$i \colon C_*^s(X;\Z)
  \longrightarrow C_*(|X|;\Z)$, given by affine linear
  parametrisation~\cite{spanier}*{Chapter~4.4 on p.~173}, induces an isomorphism on
  homology~\cite{spanier}*{Theorem~8 in Chapter~4.6 on p.~191}. 
  Moreover, $C_*^s(X;\Z)$ and $C_*(|X|;\Z)$ are free $\Z\Gamma$-chain
  complexes and $i$ is a $\Z\Gamma$-chain map. 
  Therefore,
  $i$ is a $\Z\Gamma$-chain homotopy equivalence~\cite[Theorem~(8.4) on p.~29]{brown}.
\end{proof}

Let $\nu \colon \ucov M \longrightarrow |N|$ be a $\Gamma$-map
as provided by Lemma~\ref{lem:nervemap}. We then write
\[\beta \colon C_*(\ucov M ;\Z) \longrightarrow C_*^s(N;\Z)
\]
for the composition of~$C_*(\nu;\Z) \colon C_*(\ucov M;\Z)
\longrightarrow C_*(|N|;\Z)$ and of a $\Z\Gamma$-chain homotopy
equivalence~$C_*(|N|;\Z) \longrightarrow C_*^s(N;\Z)$
(Lemma~\ref{lem:simplicialhomology}).  It should be noted that in
general $\beta$ will not be a degree-wise bounded linear map because
the number of subdivisions cannot be uniformly controlled.

Moreover, let $\alpha \colon C_*^s(N;\Z) \longrightarrow \Z[E^{*+1}]$
be the canonical $\Z\Gamma$-chain map, given by viewing~$C_*^s(N;\Z)$
as a subcomplex of~$\Z[E^{*+1}]$.

Finally, we set
\[ \varphi^s := \alpha \circ \beta \colon C_*(\ucov M;\Z) \longrightarrow \Z[E^{*+1}]
\]
and call this an \emph{equivariant subdivision}. By construction,
$\varphi^s$ is a $\Z \Gamma$-chain map. 

\begin{prop}\label{prop:subdivF}
  In the situation of Setup~\ref{setup:amenable:covers:top}, let
  $(\gamma\cdot K_i)_{(\gamma,i) \in \Gamma \times I}$ be a small
  equivariant open cover associated with~$U = (U_i)_{i \in I}$ and let $\varphi^s$
  be an equivariant subdivision as above. We 
  set
  \[   F := \bigcup_{i \in I} \{ \gamma \in \Gamma \mid \gamma \cdot K_i \cap K_i \neq \emptyset \}
  \subset \Gamma.
  \]
  Then:
  \begin{enumerate}
  \item The set~$F$ is finite.
  \item If $\mult(U) \leq n$, then
    for each singular $n$-simplex~$\sigma \colon \Delta^n
    \longrightarrow \ucov M$ and every simplex occuring
    in~$\varphi^s(\sigma)$, at least two entries have the same
    $F$-colour.

    Here, $(\gamma,i), (\lambda, j) \in E$ have the \emph{same
      $F$-colour} if $i = j$ and $\gamma^{-1}\cdot \lambda \in
    F \cap \Gamma_i$.
  \end{enumerate}
\end{prop}

\begin{proof}
  \emph{Ad~1.} As the deck transformation action~$\Gamma \curvearrowright \ucov M$
  is properly discontinuous, as $I$ is finite, and as each~$K_i$ is relatively
  compact, the set~$F$ is finite.  

  \emph{Ad~2.}
  Let $((\gamma_0,i_0), \dots, (\gamma_n,i_n))$ be an $n$-simplex that
  appears in~$\varphi^s(\sigma)$. 
  By definition of~$N$ and $C_*^s(N;\Z)$, this means that
  \[ \bigcap_{r=0}^n \gamma_r \cdot K_{i_r} \neq \emptyset.
  \]
  In particular, also $\bigcap_{r=0}^n \pi^{-1}(U_{i_r}) \neq \emptyset$.
  The multplicity of~$(\pi^{-1}(U_i))_{i \in I}$ equals the
  multiplicity of~$(U_i)_{i \in I}$, which is~$\leq n$. In
  particular, at least two of the $n+1$~elements~$i_0,\dots, i_r$ must
  be equal, say~$i_0 = i_1 =: i$.
  By definition of small equivariant open covers associated
  with~$(U_i)_{i \in I}$ (Definition~\ref{def:small:equiv:cover}),
  we have that $\gamma_0 \cdot K_i \cap \gamma_1 \cdot K_i \neq \emptyset$
  implies $\gamma_0^{-1} \cdot \gamma_1 \in \Gamma_i$.
  Moreover, we also have by definition of~$F$ that~$\gamma_0^{-1} \cdot \gamma_1 \in F$.
  Thus,
  $(\gamma_0, i_0)$ and $(\gamma_1,i_1)$ have the same $F$-colour.  
\end{proof}

%%%%%
\subsection{The simplicial volume estimate}

We combine the previously developed subdivision tools
and take advantage of asphericity to get back from the
combinatorial resolutions to the singular chain complex:

\begin{lem}\label{lem:fill}
  Let $M$ be a connected aspherical CW-complex with fundamental group~$\Gamma$
  and let $S$ be a free $\Gamma$-set. Then there exists a $\Z\Gamma$-chain
  map~$\psi \colon \Z[S^{*+1}] \to C_*(\widetilde M;\Z)$ that extends the identity~$\Z \to \Z$
  and satisfies
  \[ \|\psi\| \leq 1
  \]
  with respect to~$|\cdot|_1$.
\end{lem}
\begin{proof}
  Using asphericity of~$M$, inductively filling simplices, one can
  construct a $\Z \Gamma$-chain map~$\psi \colon \Z[S^{*+1}]
  \longrightarrow C_*(\ucov M;\Z)$ that extends~$\id_\Z$.

  This filling construction satisfies~$\|\psi\|\leq 1$ because
  each tuple is mapped to a single singular simplex.
\end{proof}

\begin{prop}\label{prop:svestimate}
  In the situation of Setup~\ref{setup:amenable:covers:top},
  let $M$ be aspherical, 
  let $\varphi^s$ be an equivariant subdivision
  as in Section~\ref{subsec:eqsubdiv}, 
  and let $z \in \Z \otimes_{\Z
    \Gamma} C_n(\ucov M;\Z)$ be a cycle.
  Moreover, let $S$ be a free $\Gamma$-set,
  let $\alpha \colon \Gamma \actson (X,\mu)$ be a free
  standard $\Gamma$-space, 
  and let $\Phi \colon
  \Z[E^{*+1}] \longrightarrow \linfz X \otimes_\Z \Z[S^{*+1}]$
  be a $\Z\Gamma$-chain map
  that extends the inclusion~$\Z \hookrightarrow \linfz X$.
  Then
  \[ \ifsvp[big]{[z]} \alpha \leq \bigl| (\Phi\circ\varphi^s)_\Gamma (z) \bigr|_1.
  \]
\end{prop}
\begin{proof}
  The methods of Section~\ref{subsec:eqsubdiv} and
  Section~\ref{sec:rokhlin} apply because the fundamental
  group~$\Gamma$ of the finite aspherical CW-complex~$M$ is
  torsion-free and countable. 
  Let $\psi \colon \Z[S^{*+1}] \to C_*(\widetilde M;\Z)$ be
  a $\Z\Gamma$-chain map that extends the identity~$\Z \to \Z$
  and satisfies~$\|\psi\|\leq 1$; such a chain map exists
  by Lemma~\ref{lem:fill}. 
  Then
  \[ \Psi := (\id_{\linfz X} \otimes_\Z \psi) \circ \Phi \circ \varphi^s
  \]
  is a $\Z\Gamma$-chain map~$C_*(\widetilde M;\Z) \to \linfz X
  \otimes_\Z C_*(\widetilde M;\Z)$ that extends the inclusion~$\Z
  \hookrightarrow \linfz X$; because $M$ is aspherical,
  $\Psi$ is thus $\Z\Gamma$-chain homotopic
  to the change of coefficients map
  (Lemma~\ref{lem:extendsunique} and Example~\ref{exa:asphres}) and the induced map
  \[ \Psi_\Gamma \colon
  \Z \otimes_{\Z\Gamma} C_*(\widetilde M;\Z)
  \to \linfz X \otimes_{\Z \Gamma} C_*(\widetilde M;\Z) 
  \]
  is chain homotopic to the change of coefficients map.
  In particular, $\Psi_\Gamma (z)$ is a cycle representing~$[z]^\alpha$;
  we thus obtain
  \begin{align*}
    \ifsvp[big]{[z]} \alpha
    & \leq \bigl| \Psi_\Gamma(z) \bigr|_1
    \leq \|\psi\| \cdot 
    \bigl| (\Phi\circ\varphi^s)_\Gamma (z) \bigr|_1
    \\
    & \leq 
    \bigl| (\Phi\circ\varphi^s)_\Gamma (z) \bigr|_1,
  \end{align*}
  as claimed.
\end{proof}

%%%%%
\subsection{The simplicial volume estimate, essential simplices}

We refine the simplicial volume estimate of
Proposition~\ref{prop:svestimate} to a bound that only
counts essential simplices and ignores
degenerate simplices. 

\begin{prop}\label{prop:essreduction}
  Let $\Gamma$ be a torsion-free group and let $S$ be a non-empty
  free $\Gamma$-set. Then there exists a $\Gamma$-equivariant
  chain map~$\eta \colon \Z[S^{*+1}] \longrightarrow \Z[S^{*+1}]$
  extending~$\id_\Z \colon \Z \to \Z$ (whence $\eta \simeq_{\Z\Gamma} \id_{\Z[S^{*+1}]}$)
  with following property: For all chains~$c \in \Z[S^{k+1}]$,
  we have
  \[ \bigl| \eta_k (c) \bigr|_1 \leq (k+1)! \cdot \essn c.
  \]
\end{prop}
\begin{proof}
  We argue similarly to case of integral singular
  homology~\cite{campagnolosauer} through a barycentric
  subdivision:
  As $S$ is non-empty, there exists a~$\sigma_0 \in S$.
  
  We write~$D$ for the set of all finite non-empty subsets of~$S$.
  Because $\Gamma$ is torsion-free, the induced action on~$D$
  is also free. Hence, $\Z[D^{*+1}]$ and $\Z[S^{*+1}]$ are
  free $\Z\Gamma$-resolutions of~$\Z$.

  First, choosing a $\Gamma$-fundamental domain~$F \subset
  D$, we have the $\Z\Gamma$-chain map~$\varrho \colon \Z[D^{*+1}]
  \longrightarrow \Z[S^{*+1}]$ defined as follows: For~$\gamma_0, \dots,
  \gamma_k \in \Gamma$ and $x_0, \dots, x_k \in F$, we set
  \[ \varrho( \gamma_0 \cdot x_0, \dots, \gamma_k \cdot x_k)
     := (\gamma_0 \cdot \sigma_0, \dots, \gamma_k \cdot \sigma_0).
  \]

  Conversely, we construct the barycentric subdivision
  map~$\delta \colon \Z[S^{*+1}] \longrightarrow \Z[D^{*+1}]$
  inductively as follows: 
  In degree~$0$, we set
  $\delta(x_0) := \{x_0\}
  $ 
  for all~$x_0 \in S$. For $z\in D$ let $(\_, z)\colon
  \Z[D^k]\to\Z[D^{k+1}]$ denote the linear extension of the map
  $D^k\to D^{k+1}$, $(y_1,\dots, y_k)\mapsto (y_1,\dots,
  y_k,z)$. Inductively, for~$k \in \N_{>0}$, we set
  \[ \delta(x_0, \dots, x_k)
  := \sum_{j=0}^k (-1)^{j+k} \cdot \bigl( \delta(x_0, \dots, \widehat x_j, \dots, x_k)
  , \{x_0, \dots, x_k\}\bigr)
  \]
  for all~$x_0, \dots, x_k \in S$. Note that in the previous formula
  $\delta(x_0, \dots, \widehat x_j, \dots, x_k)$ inductively
  is a linear combination of elements in~$D^k$, of norm at most~$k!$;
  moreover, we inductively see that $\delta$ is compatible with the
  boundary operator.
  Therefore, $\delta$ is a $\Z\Gamma$-chain map with
  \[ \| \delta_k \| \leq (k+1)!
  \]
  for all~$k \in \N$. Moreover, a straightforward inductive
  computation shows that $\delta$ maps $(\tau(x_0),\dots,\tau(x_k))$
  to $-\delta(x_0,\dots,x_k)$ for every transposition $\tau$. In
  particular, $\delta$ maps degenerate tuples to~$0$.

  We now consider the $\Z\Gamma$-chain map
  \[ \eta := \varrho \circ \delta \colon \Z[S^{*+1}] \longrightarrow \Z[S^{*+1}]
  \]
  By construction, $\eta$ extends~$\id_\Z$.
  Moreover, $\|\eta_k\| \leq (k+1)!$
  and $\eta_k$ maps degenerate tuples to~$0$. The claim follows.
\end{proof}
  
\begin{cor}\label{cor:svestimateess}
  In the situation of Setup~\ref{setup:amenable:covers:top},
  let $M$ be aspherical with torsion-free and countable fundamental
  group~$\Gamma$, 
  let $\varphi^s$ be an equivariant subdivision
  as in Section~\ref{subsec:eqsubdiv}, 
  and let $z \in \Z \otimes_{\Z
    \Gamma} C_n(\ucov M;\Z)$ be a cycle.
  Moreover, let $S$ be a free $\Gamma$-set,
  let $\alpha \colon \Gamma \actson (X,\mu)$ be a
  free standard $\Gamma$-space, 
  and 
  let $\Phi \colon
  \Z[E^{*+1}] \longrightarrow \linfz X \otimes_\Z \Z[S^{*+1}]$ be a $\Z\Gamma$-chain map
  that extends the inclusion~$\Z \hookrightarrow \linfz X$.  Then 
  \[ \ifsvp[big]{[z]} \alpha
  \leq (n+1)!
  \cdot \essn[big]{(\Phi \circ \varphi^s)_\Gamma(z)\bigr)}.
  \]
\end{cor}
\begin{proof}
  We apply Proposition~\ref{prop:svestimate} to the
  composition~$(\id_{\linfz X} \otimes_{\Z\Gamma} \eta) \circ \Phi$,
  where $\eta$ is a $\Z\Gamma$-chain
  map as provided by Proposition~\ref{prop:essreduction}
  and combine the resulting estimate with the norm estimate
  of Proposition~\ref{prop:essreduction}. 
\end{proof}

%%%%%%%%%%%%%%%%%%%%%%%%%%%%%%%%%%%%%%%%%%%%%%%%%%%%%%%%%%%%%%%%
\section{Proof of Theorem~\ref{thm:main}}\label{sec:proof}

We prove Theorem~\ref{thm:main} following the outline
of Section~\ref{sec:outline}, i.e., by combining the results from
Section~\ref{sec:rokhlin} and Section~\ref{sec:subdiv}.
More precisely, we prove the following slightly more general statement:

\begin{thm}\label{thm:maingen}
  Let $M$ be a finite connected aspherical CW-complex with 
  fundamental group~$\Gamma$, let $n \in \N$, and let $U = (U_i)_{i\in I}$ be an open amenable
  cover of~$M$ by path-connected subsets with~$\mult U \leq n$.
  Moreover, let $\Gamma \actson (X,\mu)$ be an essentially
  free standard $\Gamma$-space. 
  Let $x \in H_n(M;\Z)$. Then
  \[ \ifsvp{x} \alpha = 0.
  \]
\end{thm}
\begin{proof}
  By passing to a conull subspace, we may assume that $\Gamma \actson
  (X,\mu)$ is a free standard $\Gamma$-space.
  
  By hypothesis, in particular, we are in the situation of
  Setup~\ref{setup:amenable:covers:top}.  As $M$ is finite, $M$ is
  compact and $\Gamma$ is countable and torsion-free.  We may assume that $I$
  is finite.  Therefore, Lemma~\ref{lem:smallcovex} and
  Section~\ref{subsec:eqsubdiv} guarantee the existence of an
  equivariant subdivision~$\varphi^s$.

  Let $z \in \Z
  \otimes_{\Z\Gamma} C_n(\ucov M;\Z)$ be a cycle
  representing~$x$ in~$H_n(M;\Z)$ 
  and let
  \[ z' := (\varphi^s)_\Gamma(z)\in \Z  \otimes_{\Z\Gamma} \Z[E^{n+1}]
  \]
  be the corresponding subdivided cycle.

  For each~$\delta \in \R_{>0}$, we can construct a corresponding
  Rokhlin chain map~$\Phi^\delta \colon \Z[E^{*+1}] \to \lfsz {X \times S^{*+1}}$
  (Lemma~\ref{lem:lfsz} and Definition~\ref{rem:equivrokhlin}).
  Let
  \[ \overline \Phi^\delta := \zeta_S \circ \Phi^\delta
  \colon \Z[E^{*+1}] \to \linfz X \otimes_\Z \Z[S^{*+1}].
  \]
  Then $\overline \Phi^\delta$ is a $\Z\Gamma$-chain
  map that extends the inclusion~$\Z \hookrightarrow \linfz X$. 
  Because $M$ is aspherical, we therefore obtain
  \[ \ifsvp x \alpha
  \leq (n+1)! \cdot \essn[big]{(\overline\Phi^\delta)_\Gamma(z')}
  \]
  from Corollary~\ref{cor:svestimateess}.
  By Proposition~\ref{prop:subdivF}, every simplex in~$z'$
  satisfies the colouring condition required in Theorem~\ref{thm:rokhlinmap}.
  Therefore, Theorem~\ref{thm:rokhlinmap} shows that
  \[ \ifsvp x \alpha
  \leq (n+1)! \cdot \delta \cdot |z'|_1.
  \]
  Taking~$\delta \to 0$ gives $\ifsvp x \alpha = 0$.
\end{proof}

\begin{proof}[Proof of Theorem~\ref{thm:main}]
  We only need to note that every closed connected manifold
  has the homotopy type of a finite connected CW-complex.
  Therefore, we can apply Theorem~\ref{thm:maingen}
  to $n:=\dim M$ and $[M] \in H_n(M;\Z)$.
\end{proof}

%%%%%%%%%%%%%%%%%%%%%%%%%%%%%%%%%%%%%%%%%%%%%%%%%%%%%%%%%%%%%%%%%%%%
\section{Examples of manifolds admitting small amenable covers}\label{sec:exa}

In this section, we recall standard techniques to produce small
amenable covers and we apply them to the manifolds in Section~\ref{sec:intro:ex}.

One key ingredient is the following elementary fact:

\begin{lem}\label{lemma:pullback:cover}
Let $X$ be a connected topological space and let $Y$ be a connected CW-complex.
Suppose that there exists a continuous map $f \colon X \to Y$
whose $\pi_1$-kernel $\ker(\pi_1(f))$ is amenable. 
Then, we have 
\[
\amcat(X) \leq \dim(Y) + 1.
\]
\end{lem}
\begin{proof}
  The CW-complex~$Y$ is homotopy equivalent to a simplicial complex of the same dimension~\cite[Theorem~2C.5 on p.~182]{hatcher}. Thus we may assume that $Y$ is already a simplicial complex. 
  The space $Y$ is covered by the open stars of its vertices. The multiplicity of this cover is $\dim(Y)+1$. Each open star is contractible. In particular, 
  we have $\amcat(Y) \leq \dim(Y) + 1$. By taking the pullback along $f$ of every open amenable cover of $Y$, the 
  amenability of~$\ker \pi_1(f)$ gives~$\amcat(X) \leq
  \amcat(Y)$~\cite[Remark~2.9]{CLM}.
\end{proof}

This result allows us to compute the amenable category of a space in terms of the cohomological
dimension of certain quotients:

\begin{lem}\label{lemma:cover:cohom:dimension}
Let $n \geq 4$. Let $X$ be a connected CW-complex with fundamental group $\Gamma := \pi_1(X)$.
Suppose that $\Gamma$ contains an amenable normal subgroup~$A$
such that $\Lambda := \Gamma \slash A$ has cohomological dimension~$\cd \Lambda < n$.
Then, $\amcat(X) \leq n$.
\end{lem}
\begin{proof}
  Because $n \geq 4$ and $\cd \Lambda < n$, there exists a
  model~$Y$ of the classifying space~$B\Lambda$
  with~$\dim Y < n$~\cite[Chapter~VIII.7]{brown}. 
  We now apply Lemma~\ref{lemma:pullback:cover} to
  the composition~$B\pi \circ c_X \colon X \longrightarrow Y$,
  where $\pi \colon \Gamma \twoheadrightarrow \Lambda$ is
  the canonical projection and $c_X \colon X \to B\Gamma$
  is the classifying map.
\end{proof}

In particular, Theorem~\ref{thm:main} applies to oriented
closed connected aspherical $n$-manifolds with~$n \geq 4$,
whose fundamental group~$\Gamma$ contains an amenable
normal subgroup~$A$ with~$\cd (\Gamma/A) < n$.

The approach of constructing amenable covers as pullbacks
dates back to Gromov's proof of Yano's theorem~\cite{vbc}*{p.~41}
and was recently generalized to the case of $F$-structures by Babenko 
and Sabourau~\cite{bsfibre}*{Corollary~2.8}.

Similarly, we can handle $S^1$-foliations:

\begin{prop}\label{prop:foliation}
Let $M$ be an oriented closed connected smooth manifold that admits a regular smooth circle foliation with finite holonomy groups. 
Then, $\amcat(M) \leq \dim(M)$.
\end{prop}
\begin{proof}
Let $\mathcal{F}$ be a regular smooth circle foliation of~$M$ with finite holonomy groups, and let $X := M \slash \mathcal{F}$
be the leaf space. 
Let $\pi \colon M \to X$ be the projection map. 
Then the leaf space~$X$ is an orbifold of dimension~$\dim(X) = \dim(M)
-1$ and can be
triangulated~\citelist{\cite{moerdijk-book}*{Theorem~2.15 on
    p.~40}\cite{moerdijk+pronk}*{Proposition~1.2.1}}.

We explain how the open stars covering of a subdivision
of~$X$ leads to an amenable cover of~$M$ with small multiplicity: 

Let $n := \dim(M)$, let $L \subset M$ be a leaf of $\mathcal{F}$ and let $H$ be its holonomy group. 
Given $x \in L$, there exists a 
sufficiently small open disk~$D^{n-1}$ that is a transversal section of~$\mathcal{F}$ at~$x$~\cite{moerdijk-book}*{Section~2.3 and Remark on p.~31}
and such that each element of~$H$ can be represented by a holonomy diffeormorphism of~$D^{n-1}$.
In this situation we can define the following: If $\overline{L} \to L$ is a finite covering of~$L$ corresponding to the finite group~$H$,
we denote by~$\overline{L} \times_{H} D^{n-1}$ the quotient space of~$\overline{L} \times D^{n-1}$
under the identification~$(l h, d) \sim (l, h d)$ for all $l \in \overline{L}, h \in H, d \in D^{n-1}$.
Since each leaf is a circle, by construction, $\overline{L} \times_{H} D^{n-1}$ has the structure 
of a disk bundle over~$S^1$~\cite{moerdijk-book}*{p.~17}. 
In particular, $\overline{L} \times_{H} D^{n-1}$ has an amenable fundamental group.

By the local Reeb stability theorem~\cite{moerdijk-book}*{Theorem~2.9}, 
every leaf~$L$ of~$\mathcal{F}$ admits a saturated open neighbourhood~$V_L \subset M$
that is diffeomorphic to~$\overline{L} \times_{H} D^{n-1}$ as above;
a set $V$ is \emph{saturated} if for every~$y \in V$, the leaf passing through~$y$
is entirely contained in~$V$. 

We now consider for each leaf~$L$ of~$\mathcal{F}$ the projection~$\pi(V_L) \subset M \slash \mathcal{F}$.
By construction, the open sets~$\pi(V_L) \cong D^{n-1} \slash H$ provide an atlas for the orbifold~$M \slash \mathcal{F}$.
Hence, passing to an iterated subdivision~$T$ of the triangulation of~$X = M \slash \mathcal{F}$,
we can assume that each open star at a vertex of~$T$ is entirely contained in a
set of the form~$\pi(V_L)$ with $L$ a leaf of~$\mathcal{F}$.
Let $U$ be the open cover of~$X$ corresponding to the open stars at the vertices of~$T$
and let $U'$ be the pullback of~$U$ along~$\pi$. By construction, we have
$\mult (U') = \mult (U) =  \dim(M \slash \mathcal{F}) + 1 = \dim(M)$.
Moreover, the open cover~$U'$ is amenable because each member of~$U$
is entirely contained in some amenable set~$V_L$.
This shows that $\amcat(M) \leq \dim(M)$.
\end{proof}

As a corollary we deduce the case of smooth $S^1$-action without fixed points:

\begin{cor}\label{cor:S1:action}
Let $M$ be an oriented closed connected smooth manifold that admits a smooth $S^1$-action without fixed points. 
Then, $\amcat(M) \leq \dim(M)$.
\end{cor}
\begin{proof}
It is sufficient to notice that every smooth $S^1$-action without fixed points gives rise to a regular smooth circle foliation with finite holonomy groups~\cite[p.~16]{moerdijk-book}.
Therefore, the result is a direct consequence of Proposition~\ref{prop:foliation}.
\end{proof}

It should be noted that every smooth non-trivial
$S^1$-action on a closed aspherical manifold has no
fixed points~\cite[Corollary~1.43]{lueckl2}.

\bibliographystyle{abbrv}
\bibliography{bib}

\end{document}